\documentclass[12pt]{amsart}
\usepackage{amsmath,amssymb}
\usepackage{amsfonts}
\usepackage{amsthm}
\usepackage{latexsym}
\usepackage{graphicx}

\def\p{\partial}

\def\R{\mathbb{R}}

\def\vv<#1>{\langle#1\rangle}

\def\1{\mathbf{1}}

\def\XXint#1#2{\setbox0=\hbox{$#1{#2}{\int}$}{#2}\kern-.5\wd0 }

\def\XXint#1#2#3{{\setbox0=\hbox{$#1{#2#3}{\int}$}
     \vcenter{\hbox{$#2#3$}}\kern-.5\wd0}}

\def\vv<#1>{{\left\langle#1\right\rangle}}

\newtheorem{thm}{Theorem}[section]

\theoremstyle{definition}
\newtheorem{defn}{Definition}[section]
\theoremstyle{remark}

\numberwithin{equation}{section}

\begin{document}
\title{Extension and rigidity of Perrin's lower bound estimate for Steklov eigenvalues on graphs}
\author{Yongjie Shi$^1$}
\address{Department of Mathematics, Shantou University, Shantou, Guangdong, 515063, China}
\email{yjshi@stu.edu.cn}
\author{Chengjie Yu$^2$}
\address{Department of Mathematics, Shantou University, Shantou, Guangdong, 515063, China}
\email{cjyu@stu.edu.cn}
\thanks{$^1$Research partially supported by GDNSF with contract no. 2025A1515011144 and NNSF of China with contract no. 11701355. }
\thanks{$^2$ Research partially supported by GDNSF with contract no. 2025A1515011144 and NNSF of China with contract no. 11571215.}
\renewcommand{\subjclassname}{%
  \textup{2020} Mathematics Subject Classification}
\date{}
\begin{abstract}
In this paper, we extend a lower bound estimate for Steklov eigenvalues by Perrin \cite{Pe} on unit-weighted graphs to general weighted graphs and characterise its rigidity.
\end{abstract}
\maketitle\markboth{Shi and  Yu}{Extension and rigidity of Perrin's lower bound estimate }
\section{Introduction}
On a Riemannian manifold $M$ with boundary, the Steklov operator sends the Dirichlet boundary data of a harmonic function to its Neumann boundary data. Steklov operator is a first order nonnegative  self-adjoint elliptic pseudo-differential operator on $\Sigma:=\p M$ (see \cite[Chapter 7]{Ta}). The eigenvalues of the Steklov operator on $M$ is called the Steklov eigenvalues of $M$. Such kinds of eigenvalues were first introduced by Steklov \cite{St} when considering liquid sloshing. It was later found deep applications in geometry (see \cite{Es,FS1,FS2}) and applied mathematics (see \cite{Ku}).  For recent progresses of the topic, interested readers can consult the surveys \cite{CG24} and \cite{GP}.

In recent years, Steklov eigenvalues were introduced to discrete setting by Hua-Huang-Wang \cite{HHW} and Hassannezhad-Miclo \cite{HM} independently. Although the notion is new, there are quite a number of works considering isoperimetric estimate (see \cite{HH,HM,He-Hua1,HHW,HHW2,Pe2}), monotonicity (see \cite{He-Hua2,YY1}), Lichnerowicz estimate (see \cite{SY1,SY2}) and extremum problems (see \cite{YY2,LZ}) of Steklov eigenvalues in the discrete setting.

In this paper, we extend a lower bound for Steklov eigenvalues by Perrin \cite{Pe} on unit-weighted graphs to general weighted graphs and characterise its rigidity. Let's recall two lower bound estimates of Perrin \cite{Pe} first.

In \cite{Pe}, Perrin obtained the following interesting lower bound for graphs equipped with unit weight.
\begin{thm}[Perrin \cite{Pe}]\label{thm-Pe-1}
Let $(G,B)$ be a connected finite graph with boundary such that $E(B,B)=\emptyset$ and equipped with the unit weight. Then,
\begin{equation}\label{eq-Pe-1}
\sigma_2\geq\frac{|B|}{(|B|-1)^2d_B}
\end{equation}
where
\begin{equation}
d_B=\max\{d(x,y)\ |\ x,y\in B\}.
\end{equation}
Moreover, if the equality of \eqref{eq-Pe-1} holds, then $|B|=2$.
\end{thm}
Note that the lower bound in \eqref{eq-Pe-1} is sharp. For example, the equality of \eqref{eq-Pe-1} is attained when $G$ is path with the two end vertices the boundary vertices. For graphs equipped with general weight,  Perrin \cite{Pe} obtained the following lower bound.
\begin{thm}[Perrin \cite{Pe}]\label{thm-Pe-2}
Let $(G,B,m,w)$ be a connected weighted finite graph with boundary such that $E(B,B)=\emptyset$. Then,
\begin{equation}\label{eq-Pe-2}
\sigma_2\geq \frac{w_0}{d_BV_B}
\end{equation}
where
\begin{equation}
V_B=\sum_{x\in B}m_x\mbox{ and }w_0=\min_{e\in E(G)} w_e.
\end{equation}
\end{thm}
Note that when $G$ is equipped with the unit weight in Theorem \ref{thm-Pe-2}, the lower bound \eqref{eq-Pe-2} is weaker than \eqref{eq-Pe-1} and is not sharp. So, Theorem \ref{thm-Pe-2} does not restore the lower bound in \eqref{eq-Pe-1} when the graph is equipped with the unit weight and is not an appropriate extension of Theorem \ref{thm-Pe-1} for general weighted graphs. 

In this short note, we first obtain a more appropriate extension of Theorem \ref{thm-Pe-1} for general weighted graphs.
\begin{thm}\label{thm-Pe-g}
Let $(G,B,m,w)$ be a connected weighted finite graph with boundary. Then,
\begin{equation}\label{eq-Pe-g}
\sigma_2\geq \frac{w_0V_B}{(V_B-m_0)^2d_B}
\end{equation}
where $m_0=\min_{x\in B}m_x$.
\end{thm}
It is clear that \eqref{eq-Pe-g} is stronger than \eqref{eq-Pe-2} and when the graph is equipped with the unit weight, \eqref{eq-Pe-g} becomes \eqref{eq-Pe-1}. So, Theorem \ref{thm-Pe-g} is an appropriate extension of Theorem \ref{thm-Pe-1} to general weighted graphs.

Secondly, we characterise the rigidity of \eqref{eq-Pe-g}.
\begin{thm}\label{thm-Pe-r}
Let $(G,B,m,w)$ be a connected weighted finite graph with boundary such that
$$\sigma_2=\frac{w_0V_B}{(V_B-m_0)^2d_B}.$$
Then, we have the following conclusions:
\begin{enumerate}
\item $|B|=2$ and $m_x=m_0$ for $x\in B$;
\item There is a unique path $P:v_0\sim v_1\sim\cdots\sim v_{d_B}$ with $v_0,v_{d_B}\in B$. Moreover, $w_{v_{i-1}v_i}=w_0$ for $i=1,2,\cdots, d_B$;
\item $G$ is a comb over $P$.
\end{enumerate}
\end{thm}
Conversely, it is not hard to check that when the graph $(G,B,m,w)$ satisfies (1)--(3) in Theorem \ref{thm-Pe-r}. Then the equality of \eqref{eq-Pe-g} holds. For the definition of a comb, see Definition \ref{def-comb}.

Finally, we would like to mention that the graphs with boundary we considered in this paper are more general than those in \cite{Pe}. More precisely, we don't assume that $E(B,B)=\emptyset$ a priori.

The rest of the paper is organized as follows. In Section 2, we introduce some preliminaries on Steklov operators and Steklov eigenvalues for graphs; In Section 3, we prove Theorem \ref{thm-Pe-g} and Theorem \ref{thm-Pe-r}.
\section{Preliminary}
In this section, we introduce some preliminaries on Steklov operators and Steklov eigenvalues for graphs.

We first introduce the notion of weighted graphs with boundary.
\begin{defn}
A quadruple $(G,B,m,w)$ is called a weighted graph with boundary if
\begin{enumerate}
\item $G=(V,E)$ is a simple graph and $B\subset V$;
\item $m:V\to \R^+$ and $w:E\to \R^+$.
\end{enumerate}
The set $B$ is called the boundary of $G$, $\Omega:=V\setminus B$ is called the interior of $G$, $m$ is called the vertex-measure and $w$ is called the edge-weight.  When $m\equiv 1$ and $w\equiv 1$, $G$ is said to equip with the unit weight.
\end{defn}

For convenience, we also view $w$ as a symmetric function on $V\times V$ with
\begin{equation*}
w_{xy}=\left\{\begin{array}{ll}w(e)& e=\{x,y\}\in E\\
0&\{x,y\}\notin E
\end{array}\right.
\end{equation*}

Let $(G,B,m,w)$ be a connected weighted finite graph with boundary. Denote the  space of functions on $V$ as $A^0(G) $ and the space of skew-symmetric functions $\alpha$ on $V\times V$ such that $\alpha(x,y)=0$ when $x\not\sim y$ as $A^1(G)$. Equip $A^0(G)$ and $A^1(G)$ with the natural inner products:
\begin{equation*}
\vv<u,v>=\sum_{x\in V}u(x)v(x)m_x
\end{equation*}
and
\begin{equation*}
\vv<\alpha,\beta>=\sum_{\{x,y\}\in E}\alpha(x,y)\beta(x,y)w_{xy}=\frac12\sum_{x,y\in V}\alpha(x,y)\beta(x,y)w_{xy}
\end{equation*}
respectively. For any $u\in A^0(G)$, define the differential $du$ of $u$ as
\begin{equation*}
du(x,y)=\left\{\begin{array}{ll}u(y)-u(x)&\{x,y\}\in E\\0&\mbox{otherwise.}\end{array}\right.
\end{equation*}
Let $d^*:A^1(G)\to A^0(G)$ be the adjoint operator of $d:A^0(G)\to A^1(G)$. The Laplacian operator on $A^0(G)$ is defined as
\begin{equation*}
\Delta=-d^*d.
\end{equation*}
By direct computation,
\begin{equation*}
\Delta u(x)=\frac{1}{m_x}\sum_{y\in V}(u(y)-u(x))w_{xy}
\end{equation*}
for any $x\in V$. Moreover, by the definition of $\Delta$, it is clear that
\begin{equation}\label{eq-integration-by-part}
\vv<\Delta u,v>=-\vv<du,dv>
\end{equation}
for any $u,v\in A^0(G)$.

Moreover, for any $u\in A^0(G)$ and $x\in B$, define the outward normal derivative of $u$ at $x$ as:
\begin{equation}\label{eq-normal-derivative}
\frac{\p u}{\p n}(x):=\frac{1}{m_x}\sum_{y\in V}(u(x)-u(y))w_{xy}=-\Delta u(x).
\end{equation}
Then, by \eqref{eq-integration-by-part}, one has the following Green's formula:
\begin{equation}\label{eq-Green}
\vv<\Delta u,v>_\Omega=-\vv<du,dv>+\vv<\frac{\p u}{\p n},v>_B.
\end{equation}
Here, for any set $S\subset V$,
\begin{equation*}
\vv<u,v>_S:=\sum_{x\in S}u(x)v(x)m_x.
\end{equation*}

For each $f\in \R^B$, let $u_f$ be the harmonic extension of $f$ into $\Omega$:
\begin{equation*}
\left\{\begin{array}{ll}\Delta u_f(x)=0&x\in\Omega\\
u_f(x)=f(x)&x\in B.
\end{array}\right.
\end{equation*}
Define the Steklov operator $\Lambda:\R^B\to \R^B$ as
\begin{equation*}
\Lambda(f)=\frac{\p u_f}{\p n}.
\end{equation*}
By \eqref{eq-Green},
\begin{equation*}
\vv<\Lambda(f),g>_B=\vv<du_f,du_g>
\end{equation*}
for any $f,g\in \R^B$. This implies that $\Lambda$ is a nonnegative self-adjoint operator on $\R^B$. The eigenvalues of $\Lambda$ is called the Steklov eigenvalues of $(G,B,m,w)$. Let
\begin{equation*}
0=\sigma_1<\sigma_2\leq\cdots\leq \sigma_{|B|}
\end{equation*}
be the eigenvalues of $\Lambda$. Here, $\sigma_1=0$ because nonzero constant functions are the corresponding eigenfunctions and $\sigma_2>0$ because we  assume that $G$ is connected. It is clear that
$$\sigma_2=\min_{0\neq f\in \R^B:\vv<f,1>_B=0}\frac{\vv<du_f,du_f>}{\vv<f,f>_B}.$$
When $i>|B|$, we take the convection that $\sigma_i=+\infty$.

Finally, recall the notion of a comb (see \cite{YY1}).
\begin{defn}\label{def-comb}
Let $G$ be a connected graph and $S$ be its connected subgraph. For any $x\in S$, denote the connected component of $G-E(S)$ containing $x$ as $G_x$. If for any different vertices $x,y\in S$, $G_x\cap G_y=\emptyset$. Then, $G$ is called a comb over $S$.
\end{defn}
\section{Proofs of main results}
In this section, we prove Theorem \ref{thm-Pe-g} and Theorem \ref{thm-Pe-r}. Although the proof of Theorem \ref{thm-Pe-g} is only a slight modification of Perrin's original proof in \cite{Pe}, we will present the details here for convenience when discussing the rigidity.
\begin{proof}[Proof of Theorem \ref{thm-Pe-g}] If $|B|\leq 1$, then $\sigma_2=+\infty$. There is nothing to prove. So, assume that $|B|\geq 2$.

Let $f\in \R^B$ be an eigenfunction of $\sigma_2$. Then
\begin{equation}\label{eq-f-1}
\sum_{x\in B}f(x)m_x=\vv<f,1>_B=0.
\end{equation}
We further assume that
\begin{equation}\label{eq-f-2}
\sum_{x\in B}f^2(x)m_x=\vv<f,f>_B=1.
\end{equation}
Let $x_1\in B$ be a vertex such that
\begin{equation}\label{eq-f-3}
|f(x_1)|=\max_{x\in B} |f(x)|.
\end{equation}
We can assume that $f(x_1)>0$. Otherwise, this can be done by just replacing $f$ by $-f$. Then, by \eqref{eq-f-2}, we know that
\begin{equation}\label{eq-f-max}
f(x_1)\geq \frac{1}{\sqrt{V_B}}.
\end{equation}
Moreover, let $x_0\in B$ be a vertex such that
\begin{equation}
f(x_0)=\min_{x\in B} f(x).
\end{equation}
By \eqref{eq-f-1}, we have
\begin{equation*}
\begin{split}
-f(x_1)m_{x_1}=\sum_{x\in B\setminus\{x_1\}}f(x)m_x
\geq f(x_0)(V_B-m_{x_1}).
\end{split}
\end{equation*}
Combining this with \eqref{eq-f-max}, one has
\begin{equation}\label{eq-f-min}
f(x_0)\leq -\frac{m_{x_1}}{(V_B-m_{x_1})\sqrt{V_B}}\leq -\frac{m_0}{(V_B-m_0)\sqrt{V_B}}.
\end{equation}
Let
$$P:\ x_0=v_0\sim v_1\sim v_2\sim\cdots\sim v_l=x_1$$
be a shortest path  in $G$ connecting $x_0$ to $x_1$. It is clear that $l\leq d_B$. Then,
\begin{equation}\label{eq-sigma}
\begin{split}
\sigma_2=&\vv<du_f,du_f>\\
=&\sum_{\{x,y\}\in E}(u_f(x)-u_f(y))^2w_{xy}\\
\geq&\sum_{i=1}^l(u_f(v_i)-u_f(v_{i-1}))^2w_{v_{i-1}v_i}\\
\geq&w_0\sum_{i=1}^l(u_f(v_i)-u_f(v_{i-1}))^2\\
\geq& \frac{w_0}{l}(f(x_1)-f(x_0))^2\\
\geq&\frac{w_0V_B}{(V_B-m_0)^2d_B}
\end{split}
\end{equation}
by the Cauchy-Schwarz inequality, \eqref{eq-f-max} and \eqref{eq-f-min}. This completes the proof of the theorem.
\end{proof}
We next come to characterise the rigidity of \eqref{eq-Pe-g}.
\begin{proof}[Proof of Theorem \ref{thm-Pe-r}]
Let the notations be the same as in the proof of Theorem \ref{thm-Pe-g}. When the equality of \eqref{eq-Pe-g} holds, we know that the inequalities in \eqref{eq-sigma}
become equalities. Thus,
\begin{enumerate}
\item[(i)] $l=d_B$, $f(x_1)=\frac{1}{\sqrt V_B}$ and $f(x_0)=-\frac{m_0}{(V_B-m_0)\sqrt{V_B}}$;
\item[(ii)] $u_f(v_0)(=f(x_0)),u_f(v_1),\cdots, u_f(v_l)(=f(x_1))$ is an arithmetic progress;
\item[(iii)] $w_{v_{i-1}v_i}=w_0$ for $i=1,2,\cdots,l$;
\item[(iv)] for any $\{x,y\}\in E(G)\setminus E(P) $, $u_f(x)=u_f(y)$.
\end{enumerate}
By (iv), we know that for $i=0,1,\cdots, l$ and $x\in G_{v_i}$ where $G_{v_i}$ is the connected component of $G-E(P)$ containing $v_i$,
\begin{equation*}
u_f(x)=u_f(v_i).
\end{equation*}
By (ii), $u_{f}(v_i)\neq u_f(v_j)$ for any $0\leq i<j\leq l$. So
$$G_{v_i}\cap G_{v_j}=\emptyset$$
for any $0\leq i<j\leq l$. Thus, $G$ is a comb over $P$ and $P$ is the only path in $G$ joining $v_0(=x_0)$ and $v_l(=x_1)$.

Moreover, by (i), \eqref{eq-f-3} and  \eqref{eq-f-2}, we know that
\begin{equation*}
|f(x)|=\frac{1}{\sqrt{V_B}},\ \forall x\in B.
\end{equation*}
So
\begin{equation*}
f(x_0)=-\frac{1}{\sqrt V_B}\mbox{ and } V_B=2m_0
\end{equation*}
by (i). Thus, $B=\{x_0,x_1\}$ and $m_{x_0}=m_{x_1}=m_0$.
This completes the proof of Theorem \ref{thm-Pe-r}.
\end{proof}

\end{document}